\documentclass{article}

\usepackage{amsmath,amssymb,amsthm,mathrsfs}
\usepackage{graphicx}
\usepackage[colorlinks=true]{hyperref}
\usepackage{pdfsync}
\usepackage[top=2cm,bottom=2cm,left=2cm,right=3cm]{geometry}
\usepackage{enumitem} 
\usepackage[all]{xy}

\usepackage{tikz}\usetikzlibrary{cd}

\newtheorem{theorem}{Theorem}
\newtheorem{proposition}{Proposition}
\newtheorem{lemma}{Lemma}

\theoremstyle{definition}
\theoremstyle{definition}\newtheorem{definition}{Definition}
\theoremstyle{definition}\newtheorem{remark}{Remark}

\newcommand{\fonction}[5]{\begin{array}[t]{lrcl}#1 :&#2 &\longrightarrow &#3\\&#4& \longmapsto &#5 \end{array}}

\def\R{\mathbb{R}}
\def\N{\mathbb{N}}

\def\AC{\mathrm{AC}}
\def\PC{\mathrm{PC}}
\def\L{\mathrm{L}}
\def\K{\mathrm{K}}
\def\KK{\mathcal{K}}
\def\CC{\mathcal{C}}

\def\FF{\mathcal{F}}
\def\GG{\mathcal{G}}
\def\MM{\mathcal{M}}
\def\NN{\mathcal{N}}
\def\S{\mathcal{S}}
\def\t{\tau}
\def\bs{\backslash}
\def\di{\displaystyle}

\def\min{{\mathrm{min}}}
\def\max{{\mathrm{max}}}

\title{Unified Riccati theory for optimal permanent and sampled-data control problems in finite and infinite time horizons}

\author{Lo\"ic Bourdin\footnote{XLIM Research Institute, UMR CNRS 7252, University of Limoges, France (\texttt{loic.bourdin@unilim.fr}).} \and
Emmanuel Tr\'elat\footnote{Sorbonne Universit\'e, CNRS, Universit\'e de Paris, Inria, Laboratoire Jacques-Louis Lions (LJLL), F-75005 Paris, France (\texttt{emmanuel.trelat@sorbonne-universite.fr}).}
}

\date{}

\begin{document}

\maketitle

\begin{abstract}
We revisit and extend the Riccati theory, unifying continuous-time linear-quadratic optimal permanent and sampled-data control problems, in finite and infinite time horizons. In a nutshell, we prove that the following diagram commutes:
\begin{equation*}
\xymatrix@R=2cm@C=4cm {
\mathrm{(SD\text{-}DRE)} \hspace{-5cm} & E^{T,\Delta} \ar[r]^{T \to +\infty} \ar[d]_{\Vert \Delta \Vert \to 0} & E^{\infty,\Delta} \ar[d]^{\Vert \Delta \Vert \to 0} & \hspace{-5cm} \mathrm{(SD\text{-}ARE)} \\
\mathrm{(P\text{-}DRE)} \hspace{-5cm} & E^T  \ar[r]_{T \to +\infty} & E^\infty & \hspace{-5cm} \mathrm{(P\text{-}ARE)}
  }
\end{equation*}  
i.e., that:
\begin{itemize}
\item[--] when the time horizon $T$ tends to $+\infty$, one passes from the Sampled-Data Difference Riccati Equation~$\mathrm{(SD\text{-}DRE)}$ to the Sampled-Data Algebraic Riccati Equation~$\mathrm{(SD\text{-}ARE)}$, and from the Permanent Differential Riccati Equation~$\mathrm{(P\text{-}DRE)}$ to the Permanent Algebraic Riccati Equation~$\mathrm{(P\text{-}ARE)}$;
\item[--] when the maximal step~$\Vert \Delta \Vert$ of the time partition~$\Delta$ tends to~$0$, one passes from~$\mathrm{(SD\text{-}DRE)}$ to~$\mathrm{(P\text{-}DRE)}$, and from~$\mathrm{(SD\text{-}ARE)}$ to~$\mathrm{(P\text{-}ARE)}$.
\end{itemize}
The notation $E$ in the above diagram (with various superscripts) refers to the solution of each of the Riccati equations listed above. Our notations and analysis provide a unified framework in order to settle all corresponding results.
\end{abstract}

\medskip


\medskip



\section{Introduction}
Optimal control theory is concerned with acting on controlled dynamical systems by minimizing a given criterion. We speak of a \textit{Linear-Quadratic~(LQ)} optimal control problem when the control system is a linear differential equation and the cost is given by a quadratic integral (see \cite{Kwakernaak}). One of the main results of LQ theory is that the optimal control is expressed as a linear state feedback called \textit{Linear-Quadratic Regulator~(LQR)}. The linear state feedback is described by using the \textit{Riccati matrix} which is the solution to a nonlinear backward matrix Cauchy problem in finite time horizon (DRE: Differential Riccati Equation), and to a nonlinear algebraic matrix equation in infinite time horizon (ARE: Algebraic Riccati Equation). The LQR problem is a fundamental issue in optimal control theory. Since the pioneering works by Maxwell, Lyapunov and Kalman (see the textbooks \cite{Kwakernaak, lee1986, sontag1998}), it has been extended to many contexts, among which: discrete-time~\cite{kuvera1972}, stochastic~\cite{zhu2005}, infinite-dimensional~\cite{curtain1974}, fractional~\cite{li2008}. One of these  concerns the case where controls must be piecewise constant, which is particularly important in view of engineering applications. We speak, there, of \textit{sampled-data controls} (or \textit{digital controls}), in contrast to \textit{permanent controls}. Recall that a control problem is said to be \textit{permanent} when the control function is authorized to be modified at any time. In many problems, achieving the corresponding solution trajectory requires a permanent modification of the control. However such a requirement is not conceivable in practice for human beings, even for mechanical or numerical devices. Therefore sampled-data controls, for which only a finite number of modifications is authorized over any compact time interval, are usually considered for engineering issues. The corresponding set of \textit{sampling times} (at which the control value can be modified) is called \textit{time partition}. A vast literature deals with sampled-data control systems, as evidenced by numerous references and books (see, e.g., \cite{acker,acker2,azhm,bami,chen,fada,gero,Iser,land,nesi,raga,souz,toiv,tou} and references therein). One of the first contributions on LQ optimal sampled-data control problems can be found in~\cite{kalman1958}. This field has significantly grown since the 70's, motivated by the electrical and mechanical engineering issues with applications for example to strings of vehicles (see~\cite{astrom1963, dorato1971, levis1968, levis1971, melzer1971, middleton1990, salgado1988}). Sampled-data versions of feedback controls and of Riccati equations have been derived and, like in the fully discrete-time case (see \cite[Remark~2]{liu2014}), these two concepts in the sampled-data control case have various equivalent formulations in the literature, due to different developed approaches: in most of the references, LQ optimal sampled-data control problems are recast as fully discrete-time problems, and then the feedback control and the Riccati equation are obtained by applying the discrete-time dynamical programming principle (see \cite{bini2009,dorato1971,kalman1958}) or by applying a discrete-time version of the Pontryagin maximum principle (see \cite{astrom1963,dorato1971,kleinman1966}). 

In the present paper our objective is to provide a mathematical framework in which LQ theories in the permanent and in the sampled-data case can be settled in a unified way. We build on our recent article~\cite{bourdin2017} in which we have developed a novel approach keeping the initial continuous-time formulation of the sampled-data problem, based on a sampled-data version of the Pontryagin maximum principle (see \cite{bourdin2013,bourdin2016}). Analogies between LQ optimal permanent and sampled-data controls have already been noticed in several works (see, e.g.,~\cite{salgado1988} or~\cite[Remark~5.4]{yuz2005}). In this article we gather in a unified setting the main results of LQ optimal control theory in the following four situations: permanent / sampled-data control, finite / infinite time horizon. 
To this aim, an important tool is the map~$\FF$ defined in Section~\ref{secF}, thanks to which we formulate, in the above-mentioned four situations, feedback controls and Riccati equations in Propositions~\ref{thmriccperm}, \ref{thmriccsample}, \ref{thmriccperminf} and~\ref{thmriccsampleinf} (Sections \ref{secfinitehorizon} and \ref{secinfinitehorizon}). Moreover, exploiting the continuity of $\FF$, we establish convergence results between the involved Riccati matrices, either as the length of the time partition goes to zero or as the finite time horizon goes to infinity. Four convergence results are summarized in the diagram presented in the abstract, and we refer to our main result, Theorem~\ref{thmmain1} (stated in Section~\ref{secmain}), for the complete mathematical statement. Some of the convergence results are already known, some others are new. Hence, Theorem~\ref{thmmain1} fills some gaps in the existing literature and, in some sense, it closes the loop, which is the meaning of the commutative diagram that conveys the main message of this article.

Theorem~\ref{thmmain1} is proved in Appendix~\ref{app1}.
An important role in the proof is played by the \textit{optimizability} property (or \textit{finite cost} property), which is well known in infinite time horizon problems and is related to various notions of controllability and of stabilizability (see \cite{datta2004,terrell2009,weiss2000}). For sampled-data controls, when rewriting the original problem as a fully discrete-time problem, optimizability is formulated on the corresponding discrete-time problem (see \cite[Theorem~3]{dorato1971} or~\cite[p.~348]{levis1971}). Here, we prove in the instrumental Lemma~\ref{lemimportant} that, if the permanent optimizability property is satisfied, then the sampled-data optimizability property is satisfied for all time partitions of sufficiently small length (moreover, a bound of the minimal sampled-data cost is given, uniform with respect to the length of the time partition). This lemma plays a key role in order to prove convergence of the sampled-data Riccati matrix to the permanent one in infinite time horizon when the length of the time partition goes to zero.

\section{Preliminaries on linear-quadratic optimal control problems}\label{secprelim}

Throughout the paper, given any $p \in \N^*$, we denote by $\S^p_+$ (resp., $\S^p_{++}$) the set of all symmetric positive semi-definite (resp., positive definite) matrices of $\R^{p\times p}$. Let $n$, $m \in \N^*$, let $P \in \S^n_+$, and for every $t\in\R$, let~$A(t)\in\R^{n\times n}$, $B(t)\in\R^{n\times m}$, $Q(t)\in\S^n_+$ and $R(t)\in\S^{m}_{++}$ be matrices depending continuously on $t$. Let~$\Phi(\cdot,\cdot)$ be the \textit{state-transition matrix} (\textit{fundamental matrix solution}) associated to~$A(\cdot)$ (see \cite[Appendix~C.4]{sontag1998}).

\begin{definition}\label{defautonomous}
We speak of an \textit{autonomous setting} when $A(t)\equiv A \in \R^{n\times n}$, $B(t)\equiv B \in \R^{n\times m}$, $Q(t)\equiv Q \in \S^n_{+}$ and $R(t)\equiv R \in \S^{m}_{++}$ are constant with respect to $t$.
\end{definition}

\subsection{Notations for a unified setting}\label{secF}

In this paper we consider four different LQ optimal control problems: permanent control versus sampled-data control, and finite time horizon versus infinite time horizon. To provide a unified presentation of our results (see Propositions~\ref{thmriccperm}, \ref{thmriccsample}, \ref{thmriccperminf} and~\ref{thmriccsampleinf}), we define the map 
$$ \fonction{\FF}{\R \times \S^n_+ \times \R_+}{\R^{n\times n}}{(t,E,h)}{\FF(t,E,h) := \MM(t,E,h) \NN(t,E,h)^{-1} \MM(t,E,h)^\top - \GG(t,E,h) } $$
where $\MM(t,E,h) := \MM_1 (t,E,h) + \MM_2(t,E,h)$, $\NN(t,E,h) := \NN_1(t,E,h) + \NN_2 (t,E,h) + \NN_3(t,E,h)$ and $\GG (t,E,h) := \GG_1(t,E,h) + \GG_2(t,E,h)$, with
\begin{center}
\begin{tabular}[H]{|c|c|c|}
\hline
 & if $h > 0$ & if $h=0$ \\ \hline
& & \\ 
$\MM_1(t,E,h) := $ & $\Phi(t,t-h)^\top E \left( \dfrac{1}{h} \di \int_{t-h}^t \Phi(t,\t) B(\t) \; d\t \right)$ & $EB(t)$ \\ 
& & \\ \hline
& & \\ 
$ \MM_2(t,E,h) :=$ & $\dfrac{1}{h} \di \int_{t-h}^t \Phi (\t,t-h)^\top Q(\t) \left( \int_{t-h}^\t \Phi(\t,\xi) B(\xi) \; d\xi \right) \; d\t$ & $0_{\R^{n\times m}}$ \\ 
& & \\ \hline
& & \\ 
$ \NN_1(t,E,h) := $ & $\di \dfrac{1}{h} \int_{t-h}^t R(\t) \; d\t$ & $R(t)$ \\ 
& & \\ \hline
& & \\ 
$ \NN_2(t,E,h) := $ & $\di \dfrac{1}{h} \int_{t-h}^t \left( \int_{t-h}^\t B(\xi)^\top \Phi(\t,\xi)^\top \; d\xi \right) Q(\t) \left( \int_{t-h}^\t \Phi(\t,\xi)B(\xi) \; d\xi \right)  \; d\t $ & $0_{\R^{m\times m}}$ \\ 
& & \\ \hline
& & \\ 
$ \NN_3(t,E,h) := $ & $\di \dfrac{1}{h} \left( \int^t_{t-h} B(\t)^\top \Phi(t,\t)^\top \; d\t \right) E \left( \int^t_{t-h} \Phi(t,\t)B(\t) \; d\t \right)$ & $0_{\R^{m\times m}}$ \\ 
& & \\ \hline
& & \\ 
$ \GG_1(t,E,h) := $ & $\di \dfrac{1}{h} \int_{t-h}^t \Phi (\t,t-h)^\top Q(\t) \Phi(\t,t-h) \; d\t $ & $Q(t)$ \\ 
& & \\ \hline
& & \\ 
$ \GG_2(t,E,h) := $ & $\di  \dfrac{1}{h} \Big( \Phi(t,t-h)^\top E \Phi(t,t-h) - E \Big)$ & $A(t)^\top E + E A(t)$ \\ 
& & \\ \hline
\end{tabular}
\end{center}
The map~$\FF$ is well-defined and is continuous (see Lemma~\ref{lemF} in Appendix~\ref{appprelim}). Moreover, for $h=0$, we have
$$ \FF(t,E,0) =  EB(t)R(t)^{-1} B(t)^\top E - Q(t) - A(t)^\top E - E A(t)  \qquad \forall (t,E) \in \R \times \S^n_+. $$
One recognizes here the second member of the Permanent Differential Riccati Equation (see Proposition~\ref{thmriccperm} and Remark~\ref{remanalog}). The map~$\FF$ is designed to provide a unified notation for the permanent and sampled-data control settings.

\begin{remark}
In the \textit{autonomous setting} (see Definition~\ref{defautonomous}), the state-transition matrix is $\Phi (t,\t) = e^{(t-\t)A}$ for all~$(t,\t) \in \R \times \R$ (see, e.g.,~\cite[Lemma~C.4.1]{sontag1998}) and hence in this case the map~$\FF$ does not depend on $t$, and
$$
\FF(E,h) = \MM(E,h) \NN(E,h)^{-1} \MM(E,h)^\top - \GG(E,h)\qquad \forall E \in \S^n_+ \quad\forall h\geq 0
$$
where $\MM(E,h) := \MM_1 (E,h) + \MM_2(E,h)$, $\NN(E,h) := \NN_1(E,h) + \NN_2 (E,h) + \NN_3(E,h)$ and $\GG (E,h) := \GG_1(E,h) + \GG_2(E,h)$, with
\begin{center}
\begin{tabular}[H]{|c|c|c|}
\hline
 & if $h > 0$ & if $h=0$ \\ \hline
& & \\ 
$\MM_1(E,h) := $ & $\di e^{hA^\top} E \left( \dfrac{1}{h}  \int_{0}^h e^{\t A}  \; d\t \right) B$ & $EB$ \\ 
& & \\ \hline
& & \\ 
$ \MM_2(E,h) :=$ & $\di \dfrac{1}{h} \left( \int_{0}^h e^{\t A^\top} Q \left( \int_0^\t e^{\xi A} \; d\xi \right) \; d\t \right) B$ & $0_{\R^{n\times m}}$ \\ 
& & \\ \hline
& & \\ 
$ \NN_1(E,h) := $ & $R$ & $R$ \\ 
& & \\ \hline
& & \\ 
$ \NN_2(E,h) := $ & $\di B^\top \left( \dfrac{1}{h} \int_{0}^h \left( \int_{0}^\t e^{\xi A^\top} \; d\xi \right) Q \left( \int_{0}^\t e^{\xi A} \; d\xi \right)  \; d\t \right) B $ & $0_{\R^{m\times m}}$ \\ 
& & \\ \hline
& & \\ 
$ \NN_3(E,h) := $ & $\di B^\top \left( \dfrac{1}{h} \left( \int_0^{h} e^{\t A^\top} \; d\t \right) E \left( \int_0^{h}e^{\t A} \; d\t \right) \right) B$ & $0_{\R^{m\times m}}$ \\ 
& & \\ \hline
& & \\ 
$ \GG_1(E,h) := $ & $\di \dfrac{1}{h} \int_{0}^h e^{\t A^\top} Q e^{\t A }\; d\t  $ & $Q$ \\ 
& & \\ \hline
& & \\ 
$ \GG_2(E,h) := $ & $\di \dfrac{1}{h} \Big( e^{hA^\top} E e^{hA} - E \Big)$ & $A^\top E + E A$ \\ 
& & \\ \hline
\end{tabular}
\end{center}
In particular, in the autonomous setting and for $h=0$, we have
$$ \FF(E,0) =  EBR^{-1} B^\top E - Q - A^\top E - E A   \qquad \forall E \in \S^n_+. $$
\end{remark}

\subsection{Finite time horizon: permanent / sampled-data control}\label{secfinitehorizon}

Given any $T >0$, we denote by~$\AC([0,T],\R^n)$ the space of absolutely continuous functions defined on $[0,T]$ with values in $\R^n$, and by $\L^2([0,T],\R^m)$ the Lebesgue space of square-integrable functions defined almost everywhere on $[0,T]$ with values in $\R^m$. In what follows $\L^2([0,T],\R^m)$ is the set of \textit{permanent controls}. 


A \textit{time partition} of the interval $[0,T]$ is a finite set~$\Delta = \{ t_i \}_{i=0,\ldots,N}$, with $N \in \N^*$, such that~$ 0 = t_0 < t_1 < \ldots < t_{N-1} < t_N = T $. We denote by~$\PC^\Delta ([0,T],\R^m)$ the space of functions defined on $[0,T]$ with values in $\R^m$ that are piecewise constant according to the time partition~$\Delta$, that is
$$ 
\PC^\Delta ([0,T],\R^m) := \{ u : [0,T] \to \R^m \ \mid\ u(t) = u_i\in\R^m\quad\forall t\in[t_i,t_{i+1}), \ i=0,\ldots,N-1 \}. 
$$
In what follows $\PC^\Delta([0,T],\R^m)$ is the set of \textit{sampled-data controls} according to the time partition $\Delta$ (it is a vector space of dimension $N$).
We denote by $\Vert \Delta \Vert := \max\{ h_i,\ i=1,\ldots,N\} > 0$, where $h_i := t_i - t_{i-1} > 0$ for all~$i=1,\ldots,N$. When $h_i = h$ for some~$h > 0$ for every~$i=1,\ldots,N$, the time partition $\Delta$ is said to be \textit{$h$-uniform} (which corresponds to \textit{periodic sampling}, see~\cite[Section~II.A]{bini2014}).


In this section we consider two LQ optimal control problems in finite time horizon: permanent control~$u \in \L^2([0,T],\R^m)$ (Proposition~\ref{thmriccperm}) and sampled-data control $u \in \PC^\Delta([0,T],\R^m)$ (Proposition~\ref{thmriccsample}).

\begin{proposition}[Permanent control in finite time horizon]\label{thmriccperm}
Let $T > 0$ and let $x_0 \in \R^n$. The LQ optimal permanent control problem in finite time horizon $T$ given by
\begin{equation}\tag{$\mathrm{OCP}^T_{x_0}$}
\begin{array}{rl}
\text{minimize} &  \langle P x(T) , x(T) \rangle_{\R^n} + \di \int_0^T \Big( \langle Q(\t) x(\t) , x(\t) \rangle_{\R^n} + \langle R(\t) u(\t) , u(\t) \rangle_{\R^m} \Big) \; d\t   \\[18pt]
\text{subject to} & \left\lbrace \begin{array}{l}
x \in \AC([0,T],\R^n), \qquad u \in \L^2([0,T],\R^m)  \\[8pt]
\dot{x}(t) = A(t)x(t) + B(t)u(t) \qquad \text{for a.e.}\ t \in [0,T] \\[8pt]
x(0)=x_0
\end{array} \right.
\end{array}
\end{equation}
has a unique optimal solution $(x^*,u^*)$. Moreover $u^*$ is the time-varying state feedback
$$ u^*(t) = - \NN(t,E^T(t),0)^{-1} \MM(t,E^T(t),0)^\top x^*(t) \qquad \text{for a.e.}\ t \in [0,T] $$
where $E^T: [0,T] \to \S^{n}_+$ is the unique solution to the Permanent Differential Riccati Equation $\mathrm{(P\text{-}DRE)}$ 
\begin{equation}\tag{$\mathrm{P\text{-}DRE}$}
\left\lbrace
\begin{array}{l}
\dot{E^T}(t) = \FF(t,E^T(t),0) \qquad \forall t \in [0,T] \\[5pt]
E^T(T) = P.
\end{array}
\right.
\end{equation}
Furthermore, the minimal cost of $(\mathrm{OCP}^T_{x_0})$ is equal to $ \langle E^T(0) x_0 , x_0 \rangle_{\R^{n}}$.
\end{proposition}

\begin{proposition}[Sampled-data control in finite time horizon]\label{thmriccsample}
Let $T > 0$, let $\Delta = \{ t_i \}_{i=0,\ldots,N}$ be a time partition of the interval $[0,T]$ and let~$x_0 \in \R^n$. The LQ optimal sampled-data control problem in finite time horizon $T$ given by
\begin{equation}\tag{$\mathrm{OCP}^{T,\Delta}_{x_0}$}
\begin{array}{rl}
\text{minimize} &  \langle P x(T) , x(T) \rangle_{\R^n} + \di \int_0^T \Big(  \langle Q(\t) x(\t) , x(\t) \rangle_{\R^n} + \langle R(\t) u(\t) , u(\t) \rangle_{\R^m} \Big) \; d\t  \\[18pt]
\text{subject to} & \left\lbrace \begin{array}{l}
x \in \AC([0,T],\R^n), \qquad u \in \PC^\Delta([0,T),\R^m)  \\[8pt]
\dot{x}(t) = A(t)x(t) + B(t)u(t) \qquad \text{for a.e.}\  t \in [0,T] \\[8pt]
x(0)=x_0
\end{array} \right.
\end{array}
\end{equation}
has a unique optimal solution $(x^*,u^*)$. Moreover $u^*$ is the time-varying state feedback
$$ u^*_i = - \NN(t_{i+1},E^{T,\Delta}_{i+1},h_{i+1})^{-1} \MM(t_{i+1},E^{T,\Delta}_{i+1},h_{i+1} )^\top x^*(t_i) \qquad \forall i=0,\ldots,N-1 $$
where $E^{T,\Delta} = (E^{T,\Delta}_i)_{i=0,\ldots,N} \subset \S^{n}_+$ is the unique solution to the Sampled-Data Difference Riccati Equation~$\mathrm{(SD\text{-}DRE)}$ 
\begin{equation}\tag{$\mathrm{SD\text{-}DRE}$}
\left\lbrace
\begin{array}{l}
E^{T,\Delta}_{i+1}-E^{T,\Delta}_i = h_{i+1} \FF (t_{i+1},E^{T,\Delta}_{i+1},h_{i+1}) \qquad \forall i=0,\ldots,N-1 \\[5pt]
E^{T,\Delta}_N = P.
\end{array}
\right.
\end{equation}
Furthermore, the minimal cost of $(\mathrm{OCP}^{T,\Delta}_{x_0})$ is equal to $ \langle E^{T,\Delta}_0 x_0 , x_0 \rangle_{\R^{n}}$.
\end{proposition}

\begin{remark}\label{remanalog}
The mathematical contents of Propositions~\ref{thmriccperm} and~\ref{thmriccsample} are not new. The time-varying state feedback~$u^*$ in Proposition~\ref{thmriccperm} is usually written as
$$
u^*(t) = - R(t)^{-1} B(t)^\top E^T(t) x^*(t) \qquad \text{for a.e.}\  t \in [0,T] 
$$
and $\mathrm{(P\text{-}DRE)}$ is usually written as
\begin{equation*}
\left\lbrace
\begin{array}{l}
\dot{E^T}(t) = E^T(t) B(t) R(t)^{-1} B(t)^\top E^T(t) - Q(t) - A(t)^\top E^T(t) - E^T(t)A(t) \qquad \forall t \in [0,T] \\[5pt]
E^T(T) = P
\end{array}
\right.
\end{equation*}
(see \cite{bressan2007, Kwakernaak, lee1986, sontag1998, trelat2005}). Like in the fully discrete-time case~\cite[Remark~2]{liu2014}, the analogous results in the sampled-data control case have various equivalent formulations in the literature. 
Using the Duhamel formula, 
Problem~$(\mathrm{OCP}^{T,\Delta}_{x_0})$ can be recast as a fully discrete-time linear-quadratic optimal control problem. In this way, the time-varying state feedback control $u^*$ in Proposition~\ref{thmriccsample} and $\mathrm{(SD\text{-}DRE)}$ were first obtained in~\cite{kalman1958} by applying the discrete-time dynamical programming principle (method revisited in~\cite[p.~616]{dorato1971} or more recently in~\cite[Theorem~4.1]{bini2009}), while they are derived in~\cite[Appendix~B]{astrom1963} or in \cite[p.~618]{dorato1971} by applying a discrete-time version of the Pontryagin maximum principle (see \cite{kleinman1966}). 
In Theorem~\ref{thmmain1} hereafter, we are going to prove convergence of $E^{T,\Delta}$ to $E^T$ when $\Vert \Delta \Vert \to 0$.
\end{remark}

\subsection{Infinite time horizon: permanent / sampled-data control (autonomous setting and uniform time partition)}\label{secinfinitehorizon}
This section is dedicated to the infinite time horizon case. We denote by~$\AC([0,+\infty),\R^n)$ the space of functions defined on~$[0,+\infty)$ with values in $\R^n$ which are absolutely continuous over all intervals~$[0,T]$ with~$T> 0$, and by $\L^2([0,+\infty),\R^m)$ the Lebesgue space of square-integrable functions defined almost everywhere on $[0,+\infty)$ with values in $\R^m$. Assume that we are in the autonomous setting (see Definition~\ref{defautonomous}). We consider the following assumptions:
\begin{enumerate}
\item[$\mathrm{(H_1)}$] $Q \in \S^n_{++}$.
\item[$\mathrm{(H_2)}$] For every $x_0 \in \R^n$, there exists a pair $(x,u) \in \AC([0,+\infty),\R^n) \times \L^2 ([0,+\infty),\R^m)$ such that $\dot{x}(t) = Ax(t) + Bu(t)$ for almost every $t \geq 0$ and~$x(0)=x_0$, satisfying
$$ \di \int_0^{+\infty} \Big( \langle Q x(\t) , x(\t) \rangle_{\R^n} + \langle R u(\t) , u(\t) \rangle_{\R^m} \Big) \; d\t < +\infty. $$
\end{enumerate}
Assumption~$\mathrm{(H_2)}$ is known in the literature as \textit{optimizability} assumption (or \textit{finite cost} assumption) and is related to various notions of \textit{stabilizability} of linear permanent control systems (see \cite{weiss2000}). A wide literature is dedicated to this topic (see \cite{terrell2009} and references mentioned in \cite[Section~10.10]{datta2004}). Recall that, if the pair~$(A,B)$ satisfies the  Kalman condition (see \cite[Theorem~1.2]{zabczyk2008}) or only the weaker Popov-Belevitch-Hautus test condition (see \cite[Theorem~6.2]{terrell2009}) then~$\mathrm{(H_2)}$ is satisfied. 

\medskip

Let $h > 0$. The $h$-uniform time partition of the interval~$[0,+\infty)$ is the sequence $\Delta = \{ t_i \}_{i \in \N}$, where~$t_i := ih$ for every~$i \in \N$. We denote by $\Vert \Delta \Vert = h$ and by~$\PC^\Delta ([0,+\infty),\R^m)$ the space of functions defined on $[0,+\infty)$ with values in $\R^m$ that are piecewise constant according to the time partition $\Delta$, that is
$$ \PC^\Delta ([0,+\infty),\R^m) := \{ u  : [0,+\infty) \to \R^m\ \mid\ u(t) = u_i\quad \forall t \in [t_i,t_{i+1}), \ i\in\N  \}.   $$
We also consider the following assumption that we call \textit{$h$-optimizability} assumption:
\begin{enumerate}
\item[$\mathrm{(H}_2^h\mathrm{)}$] For every $x_0 \in \R^n$, there exists a pair $(x,u) \in \AC([0,+\infty),\R^n) \times \PC^\Delta ([0,+\infty),\R^m)$ such that $\dot{x}(t) = Ax(t) + Bu(t)$ for almost every $t \geq 0$ and $x(0)=x_0$, satisfying
$$ \int_0^{+\infty} \Big( \langle Q x(\t) , x(\t) \rangle_{\R^n} + \langle R u(\t) , u(\t) \rangle_{\R^m} \Big) \; d\t < +\infty. $$
\end{enumerate}
Obviously, if $\mathrm{(H}_2^h\mathrm{)}$ is satisfied for some $h > 0$ then $\mathrm{(H_2)}$ is satisfied. In other words, $\mathrm{(H}_2^h\mathrm{)}$ for a given $h>0$ is stronger than~$\mathrm{(H}_2\mathrm{)}$. Conversely, we will prove in Lemma~\ref{lemimportant} further that, if $\mathrm{(H_1)}$ and~$\mathrm{(H_2)}$ are satisfied, then there exists~$\overline{h} > 0$ such that $\mathrm{(H}_2^h\mathrm{)}$ is satisfied for every $h\in(0,\overline{h}]$.

\medskip

In this section, in the autonomous setting (see Definition~\ref{defautonomous}), we consider two infinite time horizon LQ optimal control problems: permanent control~$u \in \L^2([0,+\infty),\R^m)$ (Proposition~\ref{thmriccperminf}) and sampled-data control~$u \in \PC^\Delta([0,+\infty),\R^m)$ (Proposition~\ref{thmriccsampleinf}).

\begin{proposition}[Permanent control in infinite time horizon]\label{thmriccperminf}
Assume that we are in the autonomous setting (see Definition~\ref{defautonomous}). Let $x_0 \in \R^n$. Under Assumptions $\mathrm{(H_1)}$ and $\mathrm{(H_2)}$, the LQ optimal permanent control problem in infinite time horizon given by
\begin{equation}\tag{$\mathrm{OCP}^\infty_{x_0}$}
\begin{array}{rl}
\text{minimize} & \di \int_0^{+\infty} \Big( \langle Q x(\t) , x(\t) \rangle_{\R^n} + \langle R u(\t) , u(\t) \rangle_{\R^m} \Big) \; d\t  \\[18pt]
\text{subject to} & \left\lbrace \begin{array}{l}
x \in \AC([0,+\infty),\R^n), \quad u \in \L^2([0,+\infty),\R^m) \\[8pt]
\dot{x}(t) = Ax(t) + Bu(t) \qquad \text{for a.e.}\  t \geq 0 \\[8pt]
x(0)=x_0
\end{array} \right.
\end{array}
\end{equation}
has a unique optimal solution $(x^*,u^*)$. Moreover $u^*$ is the state feedback
$$ u^*(t) = - \NN(E^\infty,0)^{-1} \MM(E^\infty,0) ^\top x^*(t) \qquad \text{for a.e.}\  t \geq 0 $$
where $E^\infty \in \S^{n}_{++}$ is the unique solution to the Permanent Algebraic Riccati Equation~$\mathrm{(P\text{-}ARE)}$ 
\begin{equation}\tag{$\mathrm{P\text{-}ARE}$}
\left\lbrace
\begin{array}{l}
\FF(E^\infty,0) = 0_{\R^{n\times n}}  \\[5pt]
E^\infty \in \S^n_{+}.
\end{array}
\right.
\end{equation}
Furthermore, the minimal cost of $(\mathrm{OCP}^\infty_{x_0})$ is equal to $ \langle E^\infty x_0 , x_0 \rangle_{\R^{n}}$.
\end{proposition}

\begin{proposition}[Sampled-data control in infinite time horizon]\label{thmriccsampleinf}
Assume that we are in the autonomous setting (see Definition~\ref{defautonomous}). Let $\Delta = \{ t_i \}_{i \in \N} $ be a $h$-uniform time partition of the interval~$[0,+\infty)$ and let~$x_0 \in \R^n$. Under Assumptions~$\mathrm{(H_1)}$ and $\mathrm{(H}^h_2\mathrm{)}$, the LQ optimal sampled-data control problem in infinite time horizon given by
\begin{equation}\tag{$\mathrm{OCP}^{\infty,\Delta}_{x_0}$}
\begin{array}{rl}
\text{minimize} & \di \int_0^{+\infty} \Big( \langle Q x(\t) , x(\t) \rangle_{\R^n} + \langle R u(\t) , u(\t) \rangle_{\R^m} \Big) \; d\t  \\[18pt]
\text{subject to} & \left\lbrace \begin{array}{l}
x \in \AC([0,+\infty),\R^n), \quad u \in \PC^\Delta([0,+\infty),\R^m)  \\[8pt]
\dot{x}(t) = Ax(t) + Bu(t) \qquad \text{for a.e.}\  t \geq 0 \\[8pt]
x(0)=x_0
\end{array} \right.
\end{array}
\end{equation}
has a unique optimal solution $(x^*,u^*)$. Moreover $u^*$ is the state feedback
$$ u^*_i = - \NN(E^{\infty,\Delta},h)^{-1} \MM(E^{\infty,\Delta},h)^\top x^*(t_i) \qquad \forall i \in \N $$
where $E^{\infty,\Delta} \in \S^{n}_{++}$ is the unique solution to the Sampled-Data Algebraic Riccati Equation~$\mathrm{(SD\text{-}ARE)}$ 
\begin{equation}\tag{$\mathrm{SD\text{-}ARE}$}
\left\lbrace
\begin{array}{l}
\FF(E^{\infty,\Delta},h) = 0_{\R^{n\times n}} \\[5pt]
E^{\infty,\Delta} \in \S^n_{+}.
\end{array}
\right.
\end{equation}
Furthermore, the minimal cost of $(\mathrm{OCP}^{\infty,\Delta}_{x_0})$ is equal to $ \langle E^{\infty,\Delta} x_0 , x_0 \rangle_{\R^{n}}$.
\end{proposition}

\begin{remark}
The mathematical content of Proposition~\ref{thmriccperminf} is well known in the literature (see \cite{bressan2007, Kwakernaak, lee1986, sontag1998, trelat2005}). The state feedback control $u^*$ in Proposition~\ref{thmriccperminf} is usually written as
$$ u^*(t) = - R^{-1} B^\top E^\infty x^*(t) \qquad \text{for a.e.}\  t \geq 0 $$
and~$\mathrm{(P\text{-}ARE)}$ is usually written as
\begin{equation*}
\left\lbrace
\begin{array}{l}
E^\infty B R^{-1} B^\top E^\infty - Q - A^\top E^\infty - E^\infty A = 0_{\R^{n\times n}}  \\[5pt]
E^\infty \in \S^n_{+}.
\end{array}
\right.
\end{equation*}
As said in Remark~\ref{remanalog}, our formulation of Proposition~\ref{thmriccperminf}, using the continuous map~$\FF$ defined in Section~\ref{secF}, provides a unified presentation in the permanent and sampled-data cases. In Theorem~\ref{thmmain1} hereafter, we are going to prove convergence of $E^{\infty,\Delta}$ to $E^\infty$ when $h=\Vert \Delta \Vert \to 0$. 
\end{remark}

\begin{remark}\label{remanalog2}
Similarly to the finite time horizon case (see Remark~\ref{remanalog}), the state feedback control in Proposition~\ref{thmriccsampleinf} and $\mathrm{(SD\text{-}ARE)}$ have various equivalent formulations in the literature (see \cite{bini2014,levis1968,levis1971,melzer1971,middleton1990}) and in most of these references Problem~$(\mathrm{OCP}^{\infty,\Delta}_{x_0})$ is recast as a fully discrete-time LQ optimal control problem with infinite time horizon. In particular the optimizability property for Problem~$(\mathrm{OCP}^{\infty,\Delta}_{x_0})$ is equivalent to the optimizability of the corresponding fully discrete-time problem (see \cite[Theorem~3]{dorato1971} or~\cite[p.348]{levis1971}). In the present work we will prove that, if~$\mathrm{(H_1)}$ and~$\mathrm{(H_2)}$ are satisfied, then there exists~$\overline{h} > 0$ such that the $h$-optimizability assumption~$\mathrm{(H}_2^h\mathrm{)}$ is satisfied for every $h\in(0,\overline{h}]$ (see Lemma~\ref{lemimportant} further). Moreover, in that context, a uniform bound of the minimal cost of Problem~$(\mathrm{OCP}^{\infty,\Delta}_{x_0})$ (independently of~$h\in(0,\overline{h}]$) is obtained. It plays a key role in order to prove convergence of~$E^{\infty,\Delta}$ to~$E^\infty$ when~$h=\Vert \Delta \Vert\to 0$. 

We provide in Appendix~\ref{appthmriccsampleinf} a proof of Proposition~\ref{thmriccsampleinf} based on the $h$-optimizability assumption~$\mathrm{(H}^h_2\mathrm{)}$, by keeping the initial continuous-time formulation of Problem~$(\mathrm{OCP}^{\infty,\Delta}_{x_0})$ as in~\cite{bourdin2017}. This proof is an adaptation to the sampled-data control case of the proof of Proposition~\ref{thmriccperminf} (see \cite[p.153]{bressan2007}, \cite[Theorem~7 p.198]{lee1986} or~\cite[Theorem~4.13]{trelat2005}). Moreover it contains in particular the proof of convergence of $E^{T,\Delta}$ to $E^{\infty,\Delta}$ when~$T \to + \infty$.
\end{remark}

\section{Main result}\label{secmain}

Propositions~\ref{thmriccperm}, \ref{thmriccsample}, \ref{thmriccperminf} and~\ref{thmriccsampleinf} in Section~\ref{secprelim} give state feedback optimal controls for permanent and sampled-data LQ problems in finite and infinite time horizons. In each case, the optimal control is expressed thanks to a Riccati matrix: $E^T$, $E^{T,\Delta}$, $E^\infty$ and $E^{\infty,\Delta}$ respectively. Our main result (Theorem~\ref{thmmain1} below) asserts that the following diagram commutes:
\begin{equation*}
\xymatrix@R=2cm@C=4cm {
\mathrm{(SD\text{-}DRE)} \hspace{-5cm} & E^{T,\Delta} \ar[r]^{T \to +\infty} \ar[d]_{\Vert \Delta \Vert \to 0} & E^{\infty,\Delta} \ar[d]^{\Vert \Delta \Vert \to 0} & \hspace{-5cm} \mathrm{(SD\text{-}ARE)} \\
\mathrm{(P\text{-}DRE)} \hspace{-5cm} & E^T  \ar[r]_{T \to +\infty} & E^\infty & \hspace{-5cm} \mathrm{(P\text{-}ARE)}
  }
\end{equation*}  
The precise mathematical meaning of the above convergences is provided in the next theorem which is the main contribution of the present work. Let us first state the following lemma (proved in Appendix~\ref{applemimportant}).

\begin{lemma}\label{lemimportant}
In the autonomous setting (see Definition~\ref{defautonomous}), under Assumptions~$\mathrm{(H_1)}$ and~$\mathrm{(H_2)}$, there exist~$\overline{h} > 0$ and $\overline{c} \geq 0$ such that, for all $h$-uniform time partitions $\Delta$ of the interval~$[0,+\infty)$, with~$0 < h \leq \overline{h}$, and for every~$x_0 \in \R^n$, there exists a pair $(x,u) \in \AC([0,+\infty),\R^n) \times \PC^\Delta ([0,+\infty),\R^m)$ such that~$\dot{x}(t) = Ax(t) + Bu(t)$ for almost every $t \geq 0$ and~$x(0)=x_0$, satisfying
$$ \int_0^{+\infty} \Big( \langle Q x(\t) , x(\t) \rangle_{\R^n} + \langle R u(\t) , u(\t) \rangle_{\R^m} \Big) \; d\t \leq \overline{c} \langle E^\infty x_0 , x_0 \rangle_{\R^n} < +\infty. $$
\end{lemma}

Not only Lemma~\ref{lemimportant} asserts that, if $\mathrm{(H_1)}$ and~$\mathrm{(H_2)}$ are satisfied, then there exists~$\overline{h} > 0$ such that~$\mathrm{(H}_2^h\mathrm{)}$ is satisfied for every $h\in(0,\overline{h}]$, but it also provides a \textit{uniform} $h$-optimizability for all~$0 < h \leq \overline{h}$ (in the sense that the finite right-hand term is independent of $h$). This uniform bound plays a crucial role in order to derive convergence of $E^{\infty,\Delta}$ to $E^\infty$ when $h = \Vert \Delta \Vert \to 0$ (which corresponds to the right arrow of the above diagram and to the fourth item of Theorem~\ref{thmmain1} below). Finally, from the proof of Lemma~\ref{lemimportant} in Appendix~\ref{applemimportant}, note that a lower bound of the threshold~$\overline{h} > 0$ can be expressed in function of the norms of~$A$, $B$, $Q$, $R$ and~$E^\infty$.

\begin{theorem}[Commutative diagram]\label{thmmain1}
We have the following convergence results:
\begin{enumerate}
\item[\rm{(i)}] \textbf{Left arrow of the diagram:} Given any $T > 0$, we have
$$ \lim\limits_{  \Vert \Delta \Vert \to 0} \ \ \underset{i=0,\ldots,N}{\max} \Vert E^T(t_i) - E^{T,\Delta}_i \Vert_{\R^{n\times n}}  = 0 $$
for all time partitions $\Delta = \{ t_i \}_{i=0,\ldots,N}$ of the interval~$[0,T]$.
\item[\rm{(ii)}] \textbf{Bottom arrow of the diagram:} Assume that $P=0_{\R^{n\times n}}$ and that we are in the autonomous setting (see Definition~\ref{defautonomous}). Under Assumptions~$\mathrm{(H_1)}$ and $\mathrm{(H_2)}$, we have
$$ \lim\limits_{  T \to +\infty } E^T (t) = E^\infty  \qquad \forall t \geq 0. $$
\item[\rm{(iii)}] \textbf{Top arrow of the diagram:} Assume that $P=0_{\R^{n\times n}}$ and that we are in the autonomous setting (see Definition~\ref{defautonomous}). Let $\Delta = \{ t_i \}_{i \in \N} $ be a $h$-uniform time partition of the interval~$[0,+\infty)$. For all~$N \in \N^*$, we denote by~$\Delta_N := \Delta \cap [0,t_N]$ the~$h$-uniform time partition of the interval~$[0,t_N]$. Under Assumptions~$\mathrm{(H_1)}$ and $\mathrm{(H}^h_2\mathrm{)}$, we have
$$ \lim\limits_{  N \to +\infty } E^{t_N , \Delta_N }_i = E^{\infty , \Delta} \qquad \forall i \in \N. $$
\item[\rm{(iv)}] \textbf{Right arrow of the diagram:} In the autonomous setting (see Definition~\ref{defautonomous}), under Assumptions~$\mathrm{(H_1)}$ and $\mathrm{(H_2)}$, we have
$$ \lim\limits_{  h \to 0} E^{\infty,\Delta} = E^\infty $$
for all $h$-uniform time partitions $\Delta = \{ t_i \}_{i \in \N}$ of the interval $[0,+\infty)$ with $0 < h \leq \overline{h}$ (where $\overline{h} > 0$ is given by Lemma~\ref{lemimportant}).
\end{enumerate}
\end{theorem}

\begin{remark}
The proof of Theorem~\ref{thmmain1} is done in Appendix~\ref{appthmmain1}. Some results similar to the four items of Theorem~\ref{thmmain1} have already been discussed and can be found in the literature. For example, in the autonomous case and with $h$-uniform time partitions, the first item of Theorem~\ref{thmmain1} has been proved in~\cite[Corollary~2.3]{astrom1963} (a second-order convergence has even been derived). The second item of Theorem~\ref{thmmain1} is a well known fact and follows from the proof of Proposition~\ref{thmriccperminf} (see \cite[p.153]{bressan2007}, \cite[Theorem~7]{lee1986} or \cite[Theorem~4.13]{trelat2005}). The third item of Theorem~\ref{thmmain1} follows from the proof of Proposition~\ref{thmriccsampleinf} given in Appendix~\ref{appthmriccsampleinf} by keeping the initial continuous-time writting of Problem~$(\mathrm{OCP}^{\infty,\Delta}_{x_0})$. As evoked in Remarks~\ref{remanalog} and~\ref{remanalog2}, in the literature, the LQ optimal sampled-data control problems are usually rewritten as fully discrete-time LQ optimal control problems. As a consequence the result of the third item of Theorem~\ref{thmmain1} is usually reduced in the literature to the corresponding result at the discrete level (see \cite[Theorem~3]{dorato1971} or~\cite[p.348]{levis1971}). The last item of Theorem~\ref{thmmain1} is proved in Appendix~\ref{appthmmain1} by using the uniform $h$-optimizability obtained in Lemma~\ref{lemimportant}. Note that sensitivity analysis of $\mathrm{(SD\text{-}ARE)}$ with respect to $h$ has been explored in~\cite{fukata1979,levis1968,levis1971,melzer1971} by computing its derivative algebraically in view of optimization of the sampling period~$h$. Note that the map~$\FF$ defined in Section~\ref{secF} is a suitable candidate in order to invoke the classical implicit function theorem and justify the differentiability of~$E^{\infty,\Delta}$ with respect to~$h$. Finally the contribution of the present work is to provide a framework allowing to gather Propositions~\ref{thmriccperm}, \ref{thmriccsample}, \ref{thmriccperminf} and~\ref{thmriccsampleinf} in a unified setting, based on the continuous map~$\FF$, which moreover allows us to prove several convergence results for Riccati matrices and to summarize it in a single diagram.  
\end{remark}

\appendix

\section{Proofs}\label{app1}

Preliminaries and reminders are done in Section~\ref{appprelim}. We prove Proposition~\ref{thmriccsampleinf} in Section~\ref{appthmriccsampleinf}, Lemma~\ref{lemimportant} in Section~\ref{applemimportant} and Theorem~\ref{thmmain1} in Section~\ref{appthmmain1}.

\subsection{Preliminaries}\label{appprelim}

\begin{lemma}[A backward discrete Gr\"onwall lemma]\label{lemgronwall}
Let $N \in \N^*$ and $(w_i)_{i=0,\ldots,N}$, $(z_i)_{i=1,\ldots,N}$ and $(\mu_i)_{i=1,\ldots,N}$ be three finite nonnegative real sequences which satisfy $w_N = 0$ and 
$$ w_i \leq (1+\mu_{i+1})w_{i+1} + z_{i+1} \qquad \forall i = 0,\ldots,N-1. $$ 
Then
$$ w_i \leq \sum_{j=i+1}^N \left( \prod_{q=i+1}^{j-1} (1+\mu_q) \right) z_j \leq \sum_{j=i+1}^N e^{ \sum_{q=i+1}^{j-1} \mu_q } z_j  \qquad \forall i=0,\ldots,N-1. $$
\end{lemma}

\begin{proof}
The first inequality follows from a backward induction. The second inequality comes from the inequality $1+\mu \leq e^\mu$ for all $\mu \geq 0$.
\end{proof}

\begin{lemma}[Some reminders on symmetric matrices]\label{lemmatrice}
Let $p \in \N^*$. The following properties are satisfied:
\begin{enumerate}
\item[\rm{(i)}] Let $E \in \S^p_+$ (resp., $E \in \S^p_{++}$). Then all eigenvalues of $E$ are nonnegative (resp., positive) real numbers.
\item[\rm{(ii)}] Let $E \in \S^p_{+}$. Then $ \rho_\min(E) \Vert y \Vert_{\R^p}^2 \leq \langle Ey,y \rangle_{\R^p} \leq \rho_\max(E) \Vert y \Vert_{\R^p}^2 $ for all $y \in \R^p$, where $\rho_{\min}(E)$ and~$\rho_{\max}(E)$ stand respectively for the smallest and the largest nonnegative eigenvalues of $E$.
\item[\rm{(iii)}] Let $E \in \S^p_{++}$. Then $E$ is invertible and $E^{-1} \in \S^p_{++}$. Moreover we have~$\rho_{\min}(E^{-1}) = 1/\rho_\max(E)$ and~$\rho_{\max}(E^{-1}) = 1/\rho_\min(E)$.
\item[\rm{(iv)}] Let $E \in \S^p_+$. It holds that $\Vert E \Vert_{\R^{p \times p}} = \rho_{\max}(E) $.
\item[\rm{(v)}] Let $E \in \S^p_+$. If there exists $c \geq 0$ such that $\langle E y , y \rangle_{\R^p} \leq c \Vert y \Vert^2_{\R^p}$ for every $y \in \R^p$, then $\Vert E \Vert_{\R^{p \times p}} \leq c$.
\item[\rm{(vi)}] Let $E_1$, $E_2 \in \S^p_+$. If $\langle E_1 y , y \rangle_{\R^p} = \langle E_2 y , y \rangle_{\R^p}$ for every $y \in \R^p$ then $E_1 = E_2$.
\item[\rm{(vii)}] Let $(E_k)_{k \in \N}$ be a sequence of matrices in $ \S^p_+$. If $\langle E_k y , y \rangle_{\R^p}$ converges when $k \to +\infty$ for all~$y \in \R^p$ then~$(E_k)_{k \in \N}$ has a limit $E \in  \S^p_+$.
\end{enumerate}
\end{lemma}

\begin{proof}
The first four items are classical results (see, e.g.,~\cite{horn2013}). The fifth item follows from the fourth one. The last two items follow from the following fact: if $E \in \S^p_+$, with $E = (e_{ij})_{i,j=1,\ldots,p}$, then
$$ e_{ij} = \langle E b_j , b_i \rangle_{\R^p} = \dfrac{1}{2} \Big( \langle E (b_i+b_j) , b_i+b_j \rangle_{\R^p} - \langle E b_i , b_i \rangle_{\R^p} - \langle E b_j , b_j \rangle_{\R^p} \Big)  \qquad \forall i, j = 1, \ldots,p $$
where $\{ b_i \}_{i=1,\ldots,p}$ stands for the canonical basis of $\R^p$.
\end{proof}

\begin{lemma}[Properties of the function~$\FF$]\label{lemF}
The three following properties are satisfied:
\begin{enumerate}[label=\rm{(\roman*)}]
\item The map $\FF$ is well-defined on $\R \times \S^n_+ \times \R_+$.
\item The map $\FF$ is continuous on $\R \times \S^n_+ \times \R_+$.
\item If $\KK$ is a compact subset of $\R \times \S^n_+ \times \R_+$, then there exists a constant $c \geq 0$ such that
$$ \Vert \FF(t,E_2,h) - \FF(t,E_1,h) \Vert_{\R^{n\times n}} \leq c \Vert E_2-E_1 \Vert_{\R^{n\times n}} $$
for all $(t,E_1,E_2,h)$ such that $(t,E_1,h)  \in \KK$ and $(t,E_2,h) \in \KK$.
\end{enumerate}
\end{lemma}

\begin{proof}
{\rm (i)} For $(t,E,h) \in \R \times \S^n_+ \times \R_+$, note that $\NN_1(t,E,h) \in \S^m_{++}$, $\NN_2(t,E,h) \in \S^m_+$ and $\NN_3(t,E,h) \in \S^m_+$. Hence the sum $\NN(t,E,h) $ belongs to~$\S^m_{++}$ and thus is invertible from~(iii) of Lemma~\ref{lemmatrice}. 

{\rm (ii)} Since taking the inverse of a matrix is a continuous operation, we only need to prove that $\MM$, $\NN$ and~$\GG$ are continuous over~$\R \times \S^n_+ \times \R_+$. Let~$(t_k,E_k,h_k)_{k\in \N}$ be a sequence of $\R \times \S^n_+ \times \R_+$ which converges to some~$(t,E,h) \in \R \times \S^n_+ \times \R_+$. We need to prove that~$\MM(t_k,E_k,h_k)$, $\NN(t_k,E_k,h_k)$ and $\GG(t_k,E_k,h_k)$ converge respectively to~$\MM (t,E,h)$, $\NN(t,E,h)$ and $\GG (t,E,h)$ when~$k \to +\infty$. The case $h \neq 0$ can be treated using, for instance, the Lebesgue dominated convergence theorem. Let us discuss the case $h=0$ and let us assume, without loss of generality (since~$A$, $B$, $Q$ and $R$ are continuous matrices), that $h_k > 0$ for every $k \in \N$. In that situation we conclude by using in particular the fact that $t$ is a Lebesgue point of all integrands involved in the definitions of the functions~$\MM$, $\NN$ and $\GG$. 

{\rm (iii)} It is clear that $\FF$ is continuously differentiable over $\S^n_+$ with respect to its second variable. Similarly to the previous item, we can moreover prove that the map $(t,E,h) \mapsto \mathcal{D}_2 \FF(t,E,h)$ is continuous over~$\R \times \S^n_+ \times \R_+$. Thus the third item follows by applying the Taylor expansion formula with integral remainder.
\end{proof}

\begin{lemma}[A uniform bound for $E^T$ and $E^{T,\Delta}$]\label{lembound}
Let $T > 0$. We have
$$ \Vert E^T(t) \Vert_{\R^{n\times n}} \leq \Big( \Vert P \Vert_{\R^{n\times n}} + (T-t) \Vert Q_{|[t,T]} \Vert_{\infty} \Big) e^{2 \Vert A_{|[t,T]} \Vert_{\infty} (T-t)} \qquad \forall t \in [0,T]. $$
If $\Delta = \{ t_i \}_{i=0,\ldots,N}$ is a time partition of the interval $[0,T]$, then
$$ \Vert E^{T,\Delta}_i \Vert_{\R^{n\times n}} \leq \Big( \Vert P \Vert_{\R^{n\times n}} + (T-t_i) \Vert Q_{|[t_i,T]} \Vert_{\infty} \Big) e^{2 \Vert A_{|[t_i,T]} \Vert_{\infty} (T-t_i)} \qquad \forall i=0,\ldots,N. $$
\end{lemma}

\begin{proof}
Let us prove the first part of Lemma~\ref{lembound}. We first deal with the case $t=0$. Taking the null control in Problem~$(\mathrm{OCP}^T_y)$ and using the Duhamel formula, we deduce that its minimal cost satisfies
$$ \langle E^T(0)y,y \rangle_{\R^n} \leq \Big( \Vert P \Vert_{\R^{n\times n}} + T \Vert Q_{|[0,T]} \Vert_{\infty} \Big) e^{2T \Vert A_{|[0,T]} \Vert_{\infty} } \Vert y \Vert_{\R^n}^2  \qquad \forall y \in \R^n. $$
The result at $t=0$ then follows from (v) in Lemma~\ref{lemmatrice}. The case $0 < t < T$ can be treated similarly by considering the restriction of Problem~$(\mathrm{OCP}^T_y)$ to the time interval $[t,T]$ (instead of $[0,T]$). Finally the case~$t=T$ is obvious since $E^T(T) = P$. The second part of Lemma~\ref{lembound} is derived in a similar way.
\end{proof}

\begin{lemma}[Zero limit of finite cost trajectories at infinite time horizon]\label{leminfzero}
In the autonomous setting (see Definition~\ref{defautonomous}), under Assumption $\mathrm{(H_1)}$, for every $(x,u) \in \AC([0,+\infty),\R^n) \times \L^2 ([0,+\infty),\R^m)$ such that~$\dot{x}(t) = Ax(t) + Bu(t)$ for almost every $t \geq 0$ and satisfying
$$ \int_0^{+\infty} \Big( \langle Q x(\t) , x(\t) \rangle_{\R^n} + \langle R u(\t) , u(\t) \rangle_{\R^m} \Big) \; d\t < +\infty , $$
we have $ \lim_{t \to +\infty} x(t) = 0_{\R^n} $.
\end{lemma}

\begin{proof}
Since $Q \in \S^n_{++}$, we have $\Vert x(t) \Vert^2_{\R^n} \leq \frac{1}{\rho_\min (Q)} \langle Q x(t) , x(t) \rangle_{\R^n}$ for all $t \geq 0$. Using the assumptions we deduce that $x \in \L^2 ([0,+\infty),\R^m)$. Let us introduce $X \in \AC([0,+\infty),\R)$ defined by~$X(t) := \Vert x(t) \Vert^2_{\R^n} \geq 0$ for all $t \geq 0$. Since~$\dot{X}(t) = 2 \langle A x(t) + Bu(t), x(t) \rangle_{\R^n} $ for almost every $t \geq 0$, we deduce that~$\dot{X} \in \L^1([0,+\infty),\R)$ and thus~$X(t)$ admits a limit $\ell \geq 0$ when $t \to +\infty$. By contradiction let us assume that $\ell > 0$. Then there exists $s \geq 0$ such that $X(t) \geq \frac{\ell}{2} > 0$ for all~$t \geq s$. We get that
\begin{multline*}
\int_0^{\overline{t}} \Big( \langle Q x(\t) , x(\t) \rangle_{\R^n} + \langle R u(\t) , u(\t) \rangle_{\R^m} \Big) \; d\t \geq \rho_\min (Q) \left( \int_0^{\overline{t}} X(\t) \; d\t \right) \\
= \rho_\min (Q)  \int_0^{s} X(\t) \; d\t + \int_s^{\overline{t}} X(\t) \; d\t   
\geq \rho_\min (Q) \left( \int_0^{s} X(\t) \; d\t + (\overline{t}-s) \dfrac{\ell}{2}  \right) \qquad \forall  \overline{t} \geq s.
\end{multline*}
A contradiction is obtained by letting $\overline{t} \to + \infty$.
\end{proof}

\subsection{Proof of Proposition~\ref{thmriccsampleinf}}\label{appthmriccsampleinf}

This proof is inspired from the proof of Proposition~\ref{thmriccperminf} (see \cite[p.153]{bressan2007}, \cite[Theorem~7 p.198]{lee1986} or \cite[Theorem~4.13]{trelat2005}) and is an adaptation to the sampled-data control case. We denote by~$\Delta_N := \Delta \cap [0,t_N]$ the~$h$-uniform time partition of the interval~$[0,t_N]$ for every $N \in \N^*$.

\paragraph{Existence and uniqueness of the optimal solution.}
Let $x_0 \in \R^n$. For every $u \in \L^2([0,+\infty),\R^m)$, we denote by $x(\cdot,u) \in \AC([0,+\infty),\R^n)$ the unique solution to the Cauchy problem
$$ \left\lbrace 
\begin{array}{l}
\dot{x}(t) = A x(t) + Bu(t) \qquad \text{for a.e.}\  t \geq 0, \\[5pt]
x(0)= x_0.
\end{array}
\right. $$
We define the cost function
$$ \fonction{\CC}{\L^2([0,+\infty),\R^m)}{\R \cup \{ +\infty \} }{u}{\CC (u) := \di \int_0^{+\infty} \Big( \langle Q x(\t,u) , x(\t,u) \rangle_{\R^n} + \langle R u(\t) , u(\t) \rangle_{\R^m} \Big) \; d\t. } $$
Problem~$(\mathrm{OCP}^{\infty,\Delta}_{x_0})$ can be recast as $ \min \{ \CC(u)  \mid  u \in  \PC^\Delta ([0,+\infty),\R^m)\}$. Since $\mathrm{(H}^h_2\mathrm{)}$ is satisfied, we have
$$ \CC^* := \inf\{ \CC(u)  \mid  u \in  \PC^\Delta ([0,+\infty),\R^m)\} < +\infty. $$
Let us consider a minimizing sequence~$(u_k)_{k \in \N} \subset \PC^\Delta ([0,+\infty),\R^m)$ and, without loss of generality, we assume that $\CC (u_k) < +\infty$ for every $k \in \N$. Since $R \in \S^n_{++}$, we deduce that the sequence $(u_k)_{k \in \N}$ is bounded in~$\L^2([0,+\infty),\R^m)$ and thus, up to a subsequence (that we do not relabel), converges weakly to some $u^* \in \L^2([0,+\infty),\R^m)$. Since $\PC^\Delta ([0,+\infty),\R^m)$ is a weakly closed subspace of $\L^2([0,+\infty),\R^m)$, it follows that $u^* \in \PC^\Delta ([0,+\infty),\R^m)$. Moreover, denoting by $x_k := x(\cdot,u_k)$ for every $k \in \N$, the Duhamel formula gives
$$ x_k (t) = e^{tA} x_0 + \int_0^t e^{(t-\t)A} B u_k(\t) \; d\t  \qquad \forall t \geq 0 \qquad \forall k \in \N.  $$
By weak convergence we get that, for every $t\geq 0$, the sequence $(x_k(t))_{k \in \N}$ converges pointwise on $[0,+\infty)$ to
$$ x^* (t) := e^{tA} x_0 + \int_0^t e^{(t-\t)A} B u^*(\t) \; d\t . $$
Then, obviously, $x^* = x(\cdot,u^*)$. Moreover, by Fatou's lemma (see, e.g.,~\cite[Lemma~4.1]{brezis2011}) and by weak convergence, we get that
\begin{multline*}
\CC^* = \lim_{k \to + \infty} \CC(u_k) = \liminf_{k \to + \infty} \CC(u_k)
= \liminf_{k \to +\infty}  \int_0^{+\infty}\Big( \langle Q x_k(\t) , x_k(\t) \rangle_{\R^n} + \langle R u_k(\t) , u_k(\t) \rangle_{\R^m} \Big) \; d\t  \\
 \geq \liminf_{k \to +\infty}  \int_0^{+\infty}  \langle Q x_k(\t) , x_k(\t) \rangle_{\R^n}  \; d\t + \liminf_{k \to +\infty} \Vert u_k \Vert^2_{\L^2_R} \\
 \geq  \int_0^{+\infty} \langle Q x^*(\t) , x^*(\t) \rangle_{\R^n} \; d\t  + \Vert u^* \Vert^2_{\L^2_R} = \int_0^{+\infty} \left( \langle Q x^*(\t) , x^*(\t) \rangle_{\R^n} + \langle R u^*(\t) , u^*(\t) \rangle_{\R^m} \right) d\t = \CC(u^*)
\end{multline*}
where the norm defined by $ \Vert u \Vert_{\L^2_R} := ( \int_0^{+\infty} \langle R u(\t),u(\t) \rangle_{\R^m} \; d\t )^{1/2}$ for every $u \in \L^2([0,+\infty),\R^m)$ is equivalent to the usual one since $R \in \S^m_{++}$. We conclude that $(x^*,u^*)$ is an optimal solution to $(\mathrm{OCP}^{\infty,\Delta}_{x_0})$.

\medskip

Let us prove uniqueness. Note that $x(\cdot,\lambda u + (1-\lambda) v) = \lambda x(\cdot,u) + (1-\lambda) x(\cdot,v)$ for all $u$, $v \in \L^2([0,+\infty),\R^m)$ and all $\lambda \in [0,1]$. Hence, since moreover $Q \in \S^n_{++}$ and $R \in \S^m_{++}$, the cost function $\CC$ is strictly convex and thus the optimal solution to $(\mathrm{OCP}^{\infty,\Delta}_{x_0})$ is unique.

\paragraph{Existence of a solution to $\mathrm{(SD\text{-}ARE)}$.}
Let us introduce the sequence $(D_i )_{i \in \N} \subset \R^{n\times n}$ being the solution to the forward matrix induction given by
$$
\left\lbrace
\begin{array}{l}
D_{i+1}-D_i = - h \FF(D_{i},h) \qquad \forall i \in \N, \\[5pt]
D_0 = 0_{\R^{n\times n}}.
\end{array}
\right.
$$ 
Taking $P = 0_{\R^{n\times n}}$, one has $D_i = E^{t_N,\Delta_N}_{N-i}$ for every $i = 0, \ldots,N$ and every $N \in \N^*$. Hence the sequence~$ ( D_i )_{i \in \N} $ is well defined and is in $\S^n_+$.

\medskip

Our aim now is to prove that the sequence $ ( D_i )_{i \in \N}$ converges. Let $x_0 \in \R^n$. We denote by
$$ M := \int_0^{+\infty} \Big( \langle Q x(\t) , x(\t) \rangle_{\R^n} + \langle R u(\t) , u(\t) \rangle_{\R^m} \Big) \; d\t < +\infty $$
where $(x,u) \in \AC([0,+\infty),\R^n) \times \PC^\Delta ([0,+\infty),\R^m)$ is the pair provided in $\mathrm{(H}_2^h\mathrm{)}$. Since the minimal cost of~$(\mathrm{OCP}^{t_N,\Delta_N}_{x_0})$ (with $P = 0_{\R^{n\times n}}$) is given by $\langle E^{t_N,\Delta_N}_0 x_0 , x_0 \rangle_{\R^n} = \langle D_N x_0,x_0 \rangle_{\R^n}$ and is increasing with respect to $N$, we deduce that $\langle D_N x_0,x_0 \rangle_{\R^n}$ is increasing with respect to $N$. Since it is also bounded by $M$, we deduce that it converges when $N \to +\infty$. By~(vii) of Lemma~\ref{lemmatrice}, we conclude that the sequence $ ( D_i )_{i \in \N}$ in $\S^n_+$ converges to some $D \in \S^n_+$ which satisfies~$\FF(D,h) = 0_{\R^{n\times n}}$ by continuity of $\FF$ (see Lemma~\ref{lemF}).

\paragraph{Positive definiteness of $D$.}
Let $x_0 \in \R^n \bs \{ 0 \}$. Since $Q \in \S^n_{++}$, the minimal cost of~$(\mathrm{OCP}^{t_N,\Delta_N}_{x_0})$ (with~$P = 0_{\R^{n\times n}}$) given by $\langle E^{t_N,\Delta_N}_0 x_0 , x_0 \rangle_{\R^n} = \langle D_N x_0 , x_0 \rangle_{\R^n} $ for every~$N \in \N^*$ is positive. Since $\langle D_N x_0 , x_0 \rangle_{\R^n}$ is increasing with respect to $N$ and converges to $\langle D x_0 , x_0 \rangle_{\R^n}$, we deduce that $\langle D x_0 , x_0 \rangle_{\R^n} > 0$ and thus~$D \in \S^n_{++}$.

\paragraph{Lower bound of the minimal cost of $(\mathrm{OCP}^{\infty,\Delta}_{x_0})$.}
Our aim in this paragraph is to prove that, if~$Z \in \S^n_+$ satisfies $\FF(Z,h) = 0_{\R^{n\times n}}$, then $\langle Z x_0 , x_0 \rangle_{\R^n}$ is a lower bound of the minimal cost of $(\mathrm{OCP}^{\infty,\Delta}_{x_0})$ for every~$x_0 \in \R^n$.

Let $x_0 \in \R^n$. Let $(x,u) \in \AC([0,+\infty),\R^n) \times \PC^\Delta([0,+\infty),\R^m)$ be a pair such that~$\dot{x}(t) = Ax(t) + Bu(t)$ for almost every $t \geq 0$ and $x(0) = x_0$. Our objective is to prove that
$$ \langle Z x_0 , x_0 \rangle_{\R^n} \leq \int_{0}^{+\infty} \Big( \langle Q x(\t) , x(\t) \rangle_{\R^n} + \langle R u(\t) , u(\t) \rangle_{\R^m} \Big) \; d\t. $$
If the integral at the right-hand side is infinite, the result is obvious. Let us assume that the integral is finite. By Lemma~\ref{leminfzero}, $x(t)$ tends to $0_{\R^n}$ when $t \to +\infty$. By Proposition~\ref{thmriccsample}, the minimal cost of~$(\mathrm{OCP}^{t_N,\Delta_N}_{x_0})$ with~$P = Z$ is given by~$\langle E^{t_N,\Delta_N}_0 x_0 , x_0 \rangle_{\R^n}$ for every $N \in \N^*$. Since~$E^{t_N,\Delta_N}_N = Z$ and~$\FF(Z,h) = 0_{\R^{n\times n}}$, from the backward matrix induction, we get that $E^{t_N,\Delta_N}_i = Z$ for every $i = 0, \ldots ,N$ and every $N \in \N^*$. In particular the minimal cost of $(\mathrm{OCP}^{t_N,\Delta_N}_{x_0})$ with $P = Z$ is given by $\langle Z x_0 , x_0 \rangle_{\R^n}$ for every $N \in \N^*$. Hence
$$ \langle Z x_0 , x_0 \rangle_{\R^n} \leq \langle Z x(t_N),x(t_N) \rangle_{\R^n} + \int_0^{t_N} \Big( \langle Q x(\t) , x(\t) \rangle_{\R^n} + \langle R u(\t) , u(\t) \rangle_{\R^m} \Big) \; d\t. $$
Taking the limit $N \to +\infty$, the proof is complete.

\paragraph{Upper bound of the minimal cost of $(\mathrm{OCP}^{\infty,\Delta}_{x_0})$.}
Our aim in this paragraph is to prove that, if~$Z \in \S^n_+$ satisfies $\FF(Z,h) = 0_{\R^{n\times n}}$, then $\langle Z x_0 , x_0 \rangle_{\R^n}$ is an upper bound of the minimal cost of $(\mathrm{OCP}^{\infty,\Delta}_{x_0})$ for every~$x_0 \in \R^n$. Denote by $\MM := \MM (Z,h)$, $\NN := \NN(Z,h)$ and $\GG := \GG(Z,h)$. We similarly use the notations~$\MM_i$, $\NN_i$ and~$\GG_i$ for $i=1,2,3$ (see Section~\ref{secF} for details). 

Let $x_0 \in \R^n$. Let $x \in \AC([0,+\infty),\R^n)$ be the unique solution to
$$
\left\lbrace
\begin{array}{l}
\dot{x}(t) = A x(t) - B \NN^{-1} \MM^\top x(t_i) \qquad \text{for a.e.}\  t \in [t_i,t_{i+1}) \qquad \forall i \in \N \\[5pt]
x(0) = x_0,
\end{array}
\right.
$$ 
and let $u \in \PC^\Delta([0,+\infty),\R^m)$ defined by $u_i := - \NN^{-1} \MM^\top x(t_i)$ for every $i \in \N$. In particular~$\dot{x}(t) = Ax(t) + Bu(t)$ for almost every $t \geq 0$ and $x(0) = x_0$. 

By the Duhamel formula, we have $x(t) = (\alpha_i(t) - \beta_i(t)) x(t_i)$ for all $t \in [t_i,t_{i+1})$ and every $i \in \N$, where
$$ \alpha_i (t) := e^{(t-t_i)A} \quad \text{and} \quad \beta_i (t) := \left( \int_{t_i}^t e^{\xi A} \; d\xi \right) B \NN^{-1} \MM^\top  \qquad \forall t \in [t_i,t_{i+1}) \qquad \forall i \in \N. $$
Using the above expressions of $\alpha_i$ and $\beta_i$, and after some computations, we get that
$$ \int_{t_i}^{t_{i+1}}\Big( \langle Q x(\t) , x(\t) \rangle_{\R^n} + \langle R u(\t) , u(\t) \rangle_{\R^m} \Big) \; d\t = h \langle W_1 x(t_i) , x(t_i) \rangle_{\R^n} \qquad \forall i \in \N $$
where $ W_1 := \GG_1 + \MM \NN^{-1} \NN_2 \NN^{-1} \MM^\top - 2 \MM_2 \NN^{-1} \MM^\top + \MM \NN^{-1} \NN_1 \NN^{-1} \MM^\top $. On the other hand, using again the above expressions of $\alpha_i$ and $\beta_i$, we compute
$$ \langle Z x(t_i),x(t_i) \rangle_{\R^n} - \langle Z x(t_{i+1}),x(t_{i+1}) \rangle_{\R^n} = h \langle W_2 x(t_i),x(t_i) \rangle_{\R^n} \qquad \forall i \in \N $$
where $W_2 := - \GG_2 + 2 \MM_1 \NN^{-1} \MM^\top - \MM \NN^{-1} \NN_3 \NN^{-1} \MM$.
Using that~$\FF(Z,h) = \MM \NN^{-1} \MM^\top - \GG = 0_{\R^{n\times n}}$, we obtain $W_2 - W_1 = 0_{\R^{n\times n}}$ and thus $W_2 = W_1$. We deduce that
$$ \int_{t_i}^{t_{i+1}} \Big( \langle Q x(\t) , x(\t) \rangle_{\R^n} + \langle R u(\t) , u(\t) \rangle_{\R^m} \Big) \; d\t = \langle Z x(t_i),x(t_i) \rangle_{\R^n} - \langle Z x(t_{i+1}),x(t_{i+1}) \rangle_{\R^n} \qquad \forall i \in \N. $$
Summing these equalities and using that $Z \in \S^n_+$, we get
$$ \int_{0}^{t_{N}}\Big( \langle Q x(\t) , x(\t) \rangle_{\R^n} + \langle R u(\t) , u(\t) \rangle_{\R^m} \Big) \; d\t = \langle Z x_0 , x_0 \rangle_{\R^n} - \langle Z x(t_N) ,x(t_N) \rangle_{\R^n} \leq \langle Z x_0 , x_0 \rangle_{\R^n}  \qquad \forall N \in \N^*. $$
Passing to the limit $N \to +\infty$, we finally obtain
$$ \int_{0}^{+\infty} \Big( \langle Q x(\t) , x(\t) \rangle_{\R^n} + \langle R u(\t) , u(\t) \rangle_{\R^m} \Big) \; d\t \leq \langle Z x_0 , x_0 \rangle_{\R^n}.  $$
We deduce that $\langle Z x_0 , x_0 \rangle_{\R^n}$ is an upper bound of the minimal cost of $(\mathrm{OCP}^{\infty,\Delta}_{x_0})$ for every $x_0 \in \R^n$.

\paragraph{Minimal cost of $(\mathrm{OCP}^{\infty,\Delta}_{x_0})$ and state feedback control.} 
Let $x_0 \in \R^n$. By the previous paragraphs, since $D \in \S^n_{++} \subset \S^n_+$ satisfies $\FF(D,h) = 0_{\R^{n\times n}}$, the minimal cost of $(\mathrm{OCP}^{\infty,\Delta}_{x_0})$ is equal to $\langle D x_0 , x_0 \rangle_{\R^n}$. Moreover, by the previous paragraph, denoting by $x \in \AC([0,+\infty),\R^n)$  the unique solution to
$$
\left\lbrace
\begin{array}{l}
\dot{x}(t) = A x(t) - B \NN(D,h)^{-1} \MM(D,h)^\top x(t_i) \qquad \text{for a.e.}\  t \in [t_i,t_{i+1}) \qquad \forall i \in \N \\[5pt]
x(0) = x_0,
\end{array}
\right.
$$ 
and by $u \in \PC^\Delta([0,+\infty),\R^m)$ the control defined by $u_i := - \NN(D,h)^{-1} \MM(D,h)^\top x(t_i)$ for every $i \in \N$, we get that $\dot{x}(t) = Ax(t) + Bu(t)$ for almost every $t \geq 0$ and $x(0) = x_0$, and
$$ \int_{0}^{+\infty} \Big( \langle Q x(\t) , x(\t) \rangle_{\R^n} + \langle R u(\t) , u(\t) \rangle_{\R^m} \Big) \; d\t \leq \langle D x_0 , x_0 \rangle_{\R^n}.  $$
Since $\langle D x_0 , x_0 \rangle_{\R^n}$ is the minimal cost of $(\mathrm{OCP}^{\infty,\Delta}_{x_0})$, the above inequality is actually an equality. By uniqueness of the optimal solution~$(x^*,u^*)$, we get that $(x,u) = (x^*,u^*)$ and thus the optimal sampled-data control~$u^*$ is given by $u^*_i = - \NN(D,h)^{-1} \MM(D,h)^\top x^*(t_i)$ for every $i \in \N$.

\paragraph{Uniqueness of the solution to $\mathrm{(SD\text{-}ARE)}$.}
Assume that there exist $Z_1$, $Z_2 \in \S^n_+$ satisfying $\FF(Z_1,h) = \FF(Z_2,h) = 0_{\R^{n\times n}}$. By the previous paragraphs, the minimal cost of $(\mathrm{OCP}^{\infty,\Delta}_{x_0})$ is equal to $\langle Z_1 x_0 , x_0 \rangle_{\R^n} = \langle Z_2 x_0 , x_0 \rangle_{\R^n} $ for every $x_0 \in \R^n$. By (vi) of Lemma~\ref{lemmatrice}, we conclude that $Z_1 = Z_2$.

\paragraph{End of the proof.}
Defining $E^{\infty,\Delta} := D  \in \S^n_{++}$, the proof of Proposition~\ref{thmriccsampleinf} is complete.

\subsection{Proof of Lemma~\ref{lemimportant}}\label{applemimportant}

This proof is inspired from the techniques developed in~\cite{nesic1999} for preserving the stabilizing property of controls of nonlinear systems under sampling. We set $W := BR^{-1}B^\top E^\infty \in \R^{n\times n}$ where $E^\infty$ is given by Proposition~\ref{thmriccperminf}. Note that $E^\infty W \in \S^n_+$. Using $\mathrm{(P\text{-}ARE)}$, we obtain
$$ 2 \langle E^\infty y , (A-W)y \rangle_{\R^n} = - \langle Q y , y \rangle_{\R^n} - \langle E^\infty W y, y \rangle_{\R^m} \leq -\rho_\min (Q) \Vert y \Vert^2_{\R^n} \qquad \forall y \in \R^n $$
where $ \rho_\min (Q) > 0  $ since $Q \in \S^n_{++}$. Let $\overline{h} > 0$ be such that
$$ h \Vert A - W \Vert_{\R^{n\times n}} e^{h \Vert A \Vert_{\R^{n\times n}}} < 1  \qquad \text{and} \qquad 2 \rho_\max(E^\infty W) \dfrac{h \Vert A-W \Vert_{\R^{n\times n}} e^{h \Vert A \Vert_{\R^{n\times n}}}}{1 - h \Vert A-W \Vert_{\R^{n\times n}} e^{h \Vert A \Vert_{\R^{n\times n}}} } \leq \dfrac{\rho_\min (Q)}{2} $$
for every $h\in(0,\overline{h}]$. 

Now, let $x_0 \in \R^n$ and let $\Delta = \{ t_i \}_{i \in \N}$ be a $h$-uniform time partition of the interval $[0,+\infty)$ satisfying~$h\in(0,\overline{h}]$. Let~$x \in \AC([0,+\infty),\R^n)$ be the unique solution to
$$
\left\lbrace
\begin{array}{l}
\dot{x}(t) = A x(t) - W x(t_i) \qquad \text{for a.e.}\  t \in [t_i,t_{i+1}) \qquad \forall i \in \N \\[5pt]
x(0) = x_0,
\end{array}
\right.
$$ 
and let $u \in \PC^\Delta([0,+\infty),\R^m)$ be defined by $u_i := - R^{-1} B^\top E^\infty x(t_i)$ for every $i \in \N$. In particular~$\dot{x}(t) = Ax(t) + Bu(t)$ for almost every $t \geq 0$ and $x(0) = x_0$. 

On the one hand, we have
\begin{multline*}
\Vert x(t) - x(t_i) \Vert_{\R^n}  = \left\Vert \int_{t_i}^t \Big( Ax(\t) - W x(t_i) \Big) \; d\t \right\Vert_{\R^n} = \left\Vert \int_{t_i}^t \Big( A(x(\t)-x(t_i)) + (A- W) x(t_i) \Big) \; d\t \right\Vert_{\R^n} \\
\leq h \Vert A - W \Vert_{\R^{n\times n}} \Vert x(t_i) \Vert_{\R^n} + \Vert A \Vert_{\R^{n\times n}} \int_{t_i}^t \Vert x(\t) - x(t_i) \Vert_{\R^n} \; d\t 
\end{multline*}
and, by the Gr\"onwall lemma (see \cite[Appendix~C.3]{sontag1998}), we get that
$$ \Vert x(t) - x(t_i) \Vert_{\R^n}  \leq h \Vert A - W \Vert_{\R^{n\times n}} e^{h \Vert A \Vert_{\R^{n\times n}}} \Vert x(t_i) \Vert_{\R^n} \qquad \forall t \in [t_i,t_{i+1}) \qquad \forall i \in \N. $$
Since $\Vert x(t_i) \Vert_{\R^{n\times n}} \leq \Vert x(t)-x(t_i) \Vert_{\R^{n\times n}} + \Vert x(t) \Vert_{\R^{n\times n}}$ and $ h \Vert A - W \Vert_{\R^{n\times n}} e^{h \Vert A \Vert_{\R^{n\times n}}} < 1 $, we get that
$$ \Vert x(t_i) \Vert_{\R^{n\times n}} \leq \dfrac{1}{1-h \Vert A - W \Vert_{\R^{n\times n}} e^{h \Vert A \Vert_{\R^{n\times n}}}} \Vert x(t) \Vert_{\R^{n\times n}},   \qquad \forall t \in [t_i,t_{i+1}) \qquad \forall i \in \N $$
and thus
$$ \Vert x(t) - x(t_i) \Vert_{\R^{n\times n}} \leq \dfrac{h \Vert A - W \Vert_{\R^{n\times n}} e^{h \Vert A \Vert_{\R^{n\times n}}}}{1-h \Vert A - W \Vert_{\R^{n\times n}} e^{h \Vert A \Vert_{\R^{n\times n}}}} \Vert x(t) \Vert_{\R^{n\times n}},  \qquad \forall t \in [t_i,t_{i+1}) \qquad \forall i \in \N.  $$
On the other hand, we have
\begin{multline*}
\dfrac{d}{dt} \langle E^\infty x(t),x(t) \rangle_{\R^n} = 2 \langle E^\infty x(t) , \dot{x}(t) \rangle_{\R^n} 
= 2 \langle E^\infty x(t) , A x(t) - W x(t_i) \rangle_{\R^n} \\
= 2 \langle E^\infty x(t) , (A-W) x(t)  \rangle_{\R^n} + 2 \langle E^\infty x(t) , W (x(t)-x(t_i))  \rangle_{\R^n}  \qquad \text{for a.e.}\  t \in [t_i,t_{i+1}) \qquad \forall i \in \N.
\end{multline*}
We deduce that
\begin{multline*}
\dfrac{d}{dt}  \langle E^\infty x(t),x(t) \rangle_{\R^n}
\leq \left( - \rho_\min (Q) + 2 \rho_\max(E^\infty W) \dfrac{h \Vert A-W \Vert_{\R^{n\times n}} e^{h \Vert A \Vert_{\R^{n\times n}}}}{1 - h \Vert A-W \Vert_{\R^{n\times n}} e^{h \Vert A \Vert_{\R^{n\times n}}} } \right) \Vert x(t) \Vert^2_{\R^n}  \\
\leq - \dfrac{\rho_\min (Q)}{2}  \Vert x(t) \Vert^2_{\R^n} \leq - \dfrac{\rho_\min (Q)}{2 \rho_\max(E^\infty) } \langle E^\infty x(t) , x(t) \rangle_{\R^n} \qquad \text{for a.e.}\  t \geq 0.
\end{multline*}
We deduce from the Gr\"onwall lemma that
$$ \Vert x(t) \Vert^2_{\R^n} \leq \dfrac{1}{\rho_\min (E^\infty)} \langle E^\infty x(t) , x(t) \rangle_{\R^n} \leq \dfrac{1}{\rho_\min (E^\infty)} \langle E^\infty x_0 , x_0 \rangle_{\R^n} e^{- \frac{\rho_\min (Q)}{2 \rho_\max(E^\infty) } t} \qquad \forall t \geq 0. $$ 
We deduce that
\begin{multline*}
 \int_0^{+\infty} \langle Q x(\t) , x(\t) \rangle_{\R^n} \; d\t \leq \dfrac{\rho_\max(Q)}{\rho_\min (E^\infty)} \langle E^\infty x_0 , x_0 \rangle_{\R^n} \int_0^{+\infty} e^{- \frac{\rho_\min (Q)}{2 \rho_\max(E^\infty) } \t} \; d\t \\
 = \dfrac{2 \rho_\max(Q) \rho_\max(E^\infty)}{\rho_\min (Q) \rho_\min (E^\infty)} \langle E^\infty x_0 , x_0 \rangle_{\R^n} < +\infty. 
\end{multline*}
Moreover, using that $t_i = ih$ for every $i \in \N$, we have
\begin{multline*}
 \int_0^{+\infty} \langle R u(\t) , u(\t) \rangle_{\R^n} \; d\t \leq h \rho_\max(R) \sum_{i \in \N} \Vert u_i \Vert^2_{\R^m} \leq h \rho_\max(R) \Vert R^{-1} B^\top E^\infty \Vert^2_{\R^{m \times n}} \sum_{i \in \N} \Vert x(t_i) \Vert^2_{\R^m} \\
  \leq h \dfrac{\rho_\max(R)}{\rho_\min (E^\infty)}  \Vert R^{-1} B^\top E^\infty \Vert^2_{\R^{m \times n}}  \langle E^\infty x_0 , x_0 \rangle_{\R^n} \sum_{i \in \N} \Big( e^{- \frac{\rho_\min (Q)}{2 \rho_\max(E^\infty) } h} \Big)^i  \\
  = h \dfrac{\rho_\max(R)}{\rho_\min (E^\infty)}  \Vert R^{-1} B^\top E^\infty \Vert^2_{\R^{m \times n}}  \langle E^\infty x_0 , x_0 \rangle_{\R^n} \dfrac{1}{1 - e^{- \frac{\rho_\min (Q)}{2 \rho_\max(E^\infty) } h}} \\
  \leq  \dfrac{2 \rho_\max(R) \rho_\max(E^\infty)}{\rho_\min ( Q ) \rho_\min (E^\infty)}  \Vert R^{-1} B^\top E^\infty \Vert^2_{\R^{m \times n}}  \langle E^\infty x_0 , x_0 \rangle_{\R^n} e^{\frac{\rho_\min (Q)}{2 \rho_\max(E^\infty) } \overline{h}}  < +\infty.
\end{multline*}
Taking
$$ \overline{c} := \dfrac{2 \rho_\max(E^\infty)}{\rho_\min ( Q ) \rho_\min (E^\infty)} \Big( \rho_\max ( Q ) +\rho_\max ( R )  \Vert R^{-1} B^\top E^\infty \Vert^2_{\R^{m \times n}} e^{\frac{\rho_\min (Q)}{2 \rho_\max(E^\infty) } \overline{h}} \Big) \geq 0 ,$$
the proof is complete.

\subsection{Proof of Theorem~\ref{thmmain1}}\label{appthmmain1}

\paragraph*{First item.}
This proof is inspired from the classical Lax theorem in numerical analysis (see \cite[p.73]{polyanin2018}). Let~$\varepsilon > 0$. We define the map
$$ \fonction{\varphi}{[0,T] \times [0,T]}{\R^{n\times n}}{(t,h)}{\varphi(t,h) := \FF(t,E^T(t),h).} $$
By continuity of $E^T$ on $[0,T]$ and by Lemma~\ref{lemF}, the map $\varphi$ is uniformly continuous on the compact set~$[0,T] \times [0,T]$. Hence there exists $\delta > 0$ such that
$$ \Vert \varphi(t_2,h_2) - \varphi(t_1,h_1) \Vert_{\R^{n\times n}} \leq \dfrac{\varepsilon}{2Te^{cT}} $$
for all $(t_1,h_1)$, $ (t_2,h_2) \in [0,T] \times [0,T]$ satisfying $\vert t_2-t_1 \vert + \vert h_2 - h_1 \vert \leq \delta$, where $c \geq 0$ is the constant given in Lemma~\ref{lemF} associated to the compact set $\KK := [0,T] \times \K \times [0,T]$ where
$$ \K := \left\lbrace E \in \S^n_+ \mid \Vert E \Vert_{\R^{n\times n}} \leq \Big( \Vert P \Vert_{\R^{n\times n}} + T \Vert Q_{|[0,T]} \Vert_{\infty} \Big) e^{2T \Vert A_{|[0,T]} \Vert_{\infty} }  \right\rbrace . $$
In the sequel we consider a time partition $\Delta = \{ t_i \}_{i=0,\ldots,N}$ of the interval~$[0,T]$ such that $0 < \Vert \Delta \Vert \leq \delta$. Note that
\begin{multline*}
 E^{T,\Delta}_i = E^{T,\Delta}_{i+1} - h_{i+1} \FF (t_{i+1},E^{T,\Delta}_{i+1},h_{i+1}) \\
  \text{and} \quad E^T(t_i) = E^T(t_{i+1}) - h_{i+1} \FF(t_{i+1},E^T(t_{i+1}),h_{i+1}) +\eta_{i+1} \qquad \forall i=0,\ldots,N-1
\end{multline*}
where
$$ \eta_{i+1} := E^T(t_i)-E^T(t_{i+1}) + h_{i+1} \FF (t_{i+1},E^T(t_{i+1}),h_{i+1}) \qquad \forall i=0,\ldots,N-1 .$$
By Lemmas~\ref{lemF} and~\ref{lembound}, we have
$$ \Vert E^T(t_i)-E^{T,\Delta}_i \Vert_{\R^{n\times n}} \leq (1+c h_{i+1} ) \Vert E^T(t_{i+1})-E^{T,\Delta}_{i+1} \Vert_{\R^{n\times n}} + \Vert \eta_{i+1} \Vert_{\R^{n\times n}} \qquad \forall i=0,\ldots,N-1 .$$
It follows from the backward discrete Gr\"onwall lemma (see Lemma~\ref{lemgronwall}) that
$$ \Vert E^T(t_i)-E^{T,\Delta}_i \Vert_{\R^{n\times n}} \leq \sum_{j=i+1}^N e^{ c \sum_{q=i+1}^{j-1} h_q } \Vert \eta_{j} \Vert_{\R^{n\times n}} \leq e^{cT} \sum_{j=1}^N \Vert \eta_{j} \Vert_{\R^{n\times n}}  \qquad \forall i=0,\ldots,N-1 . $$
Since 
\begin{multline*}
\eta_{j} = h_j \Big( \FF(t_j,E^T(t_j),h_j) - \FF(t_j,E^T(t_j),0) \Big) 
+ \int_{t_{j-1}}^{t_j} \left( \FF(t_j,E^T(t_j),0) - \FF(\t,E^T(\t),0) \right) d\t \\
= h_j \Big( \varphi(t_j,h_j) - \varphi(t_j,0) \Big) + \int_{t_{j-1}}^{t_j} \left( \varphi(t_j,0) - \varphi(\t,0) \right) d\t \qquad \forall j=1,\ldots,N
\end{multline*}
we obtain, by uniform continuity of $\varphi$ and using that $0 < \Vert \Delta \Vert \leq \delta$, 
$$ \Vert \eta_{j} \Vert_{\R^{n\times n}} \leq 2 h_j \dfrac{\varepsilon }{2Te^{cT}} = h_j \dfrac{\varepsilon }{Te^{cT}} \qquad \forall j=1,\ldots,N. $$
We conclude that
$$ \Vert E^T(t_i)-E^{T,\Delta}_i \Vert_{\R^{n\times n}} \leq e^{cT} \sum_{j=1}^N \Vert \eta_{j} \Vert_{\R^{n\times n}} \leq e^{cT} \sum_{j=1}^N h_j \dfrac{\varepsilon }{Te^{cT}} = \dfrac{\varepsilon}{T} \sum_{j=1}^N h_j = \varepsilon  \qquad \forall i=0,\ldots,N-1. $$
The proof is complete.

\paragraph*{Second item.}
The second item of Theorem~\ref{thmmain1} is well known  and follows from the proof of Proposition~\ref{thmriccperminf} (see \cite[p.153]{bressan2007}, \cite[Theorem~7]{lee1986} or \cite[Theorem~4.13]{trelat2005}).

\paragraph*{Third item.}
This result follows from the proof of Proposition~\ref{thmriccsampleinf}. Indeed, using the notations from Appendix~\ref{appthmriccsampleinf}, it is clear that 
$$ \lim\limits_{N \to + \infty} E^{t_N,\Delta_N}_i = \lim\limits_{N \to + \infty} D_{N-i} = D = E^{\infty,\Delta} \qquad \forall i \in \N. $$

\paragraph*{Fourth item.}
By contradiction let us assume that $E^{\infty,\Delta}$ does not converge to $E^\infty$ when $h \to 0$. Then there exists $\varepsilon > 0$ and a positive sequence $(h_k)_{k \in \N}$ converging to $0$ such that $\Vert E^{\infty,\Delta_k} - E^\infty \Vert_{\R^{n\times n}} \geq \varepsilon$ for every $k \in \N$, where $\Delta_k$ stands for the $h_k$-uniform time partition of the interval $[0,+\infty)$. Without loss of generality, we assume that $0 < h_k \leq \overline{h}$ for every $k \in \N$. It follows from Proposition~\ref{thmriccsampleinf} and from Lemma~\ref{lemimportant} that the minimal cost of $(\mathrm{OCP}^{\infty,\Delta_k}_{x_0})$ satisfies 
$$ \langle E^{\infty,\Delta_k} x_0 , x_0 \rangle_{\R^{n}} \leq \overline{c} \langle E^{\infty} x_0 , x_0 \rangle_{\R^{n}} \leq \overline{c} \Vert E^\infty \Vert_{\R^{n\times n} } \Vert x_0 \Vert_{\R^n}^2  \qquad \forall x_0 \in \R^n. $$ 
Hence $\Vert E^{\infty,\Delta_k} \Vert_{\R^{n\times n}} \leq \overline{c} \Vert E^{\infty} \Vert_{\R^{n\times n}}$ for every $k \in \N$ by (v) of Lemma~\ref{lemmatrice}. Thus the sequence~$(E^{\infty,\Delta_k})_{k \in \N}$ is bounded in $\R^{n\times n}$ and, up to a subsequence (that we do not relabel), converges to some~$L \in \R^{n\times n}$. In particular $\Vert L - E^\infty \Vert_{\R^{n\times n}} \geq \varepsilon$. Since $E^{\infty,\Delta_k} \in \S^n_{++} \subset \S^n_+$ for every~$k \in \N$, it is clear that~$L \in \S^n_+$. Moreover, by $\mathrm{(SD\text{-}ARE)}$ associated to $h_k$ (see Proposition~\ref{thmriccsampleinf}), we know that~$\FF(E^{\infty,\Delta_k},h_k) = 0_{\R^{n\times n}}$ for all $k \in \N$. By continuity of $\FF$ (see Lemma~\ref{lemF}), we conclude that~$\FF(L,0) = 0_{\R^{n\times n}}$. By uniqueness (see Proposition~\ref{thmriccperminf}) we deduce that $L = E^\infty$ which raises a contradiction with the inequality $\Vert L - E^\infty \Vert_{\R^{n\times n}} \geq \varepsilon$. The proof is complete.


\end{document}